\documentclass[12pt]{article}
\usepackage{amssymb}
\usepackage{mathrsfs}
\usepackage{amsfonts}
\usepackage{tikz}
\usepackage{titlesec}
\usepackage[titletoc]{appendix}
\usepackage{amsmath,amsthm,amsfonts,amssymb,amscd}

\allowdisplaybreaks[4]
\newtheorem {Lemma}{Lemma}[section]
\newtheorem{Corollary}[Lemma]{Corollary}
\newtheorem {Theorem} [Lemma]{Theorem}

\begin{document}

\title{  A family of graphs that are determined by their normalized Laplacian spectra \thanks{This work is supported by  the Joint NSFC-ISF Research Program (jointly funded by the National Natural Science Foundation of China and the Israel Science Foundation (Nos. 11561141001, 2219/15),  the National Natural Science Foundation of China (No.11531001).
}}

\author{ Abraham Berman,\\
Department of Mathematics, \\
Technion-Israel Institute of Technology, Haifa 32000, Israel\\
Dong-Mei Chen, Zhi-Bing  Chen\\
College of Mathematics and Statistics, \\
 Shenzhen University, Shenzhen 518060,  P.R. China\\
  Wen-Zhe Liang,  Xiao-Dong Zhang\footnote{Corresponding author. E-mail: xiaodong@sjtu.edu.cn}\\
School of Mathematical Sciences, MOE-LSC, SHL-MAC\\
Shanghai Jiao Tong University,
Shanghai 200240, P. R. China}


\date{}
\maketitle

\begin{abstract}

Let  $F_{p,q}$ be the generalized friendship graph $K_1\bigvee (pK_q)$ on $pq+1$ vertices obtained by joining a vertex to all vertices of $p$ disjoint copies of the complete graph $K_q$ on $q$ vertices. In this paper, we prove that $F_{p,q}$ is determined by its normalized Laplacian spectrum  if and only if $q\ge 2$, or $q=1$ and $p\le 2$.
\end{abstract}

\section{Introduction}

 Spectral graph theory studies  the relations between the structure of a graph and eigenvalues of  matrices associated with it.
  One of the main problems in  spectral graph theory is  which graphs are determined by their spectrum or equivalently, finding nonisomorphic graphs
 $G$ and $H$ that have the same  spectrum.  Many results on these questions can be found in two excellent surveys   (see \cite{van2003,van2009}) by
  Van Dam and Haemers.

 Let $G $ be an   undirected  simple  graph with vertex set
$V(G)
$ and edge set $E(G)$. Let $d_v$ be the degree of a vertex $v\in V(G)$.
 The {\it normalized Laplacian matrix} of a graph $G$ is defined to be
 $\mathcal{L}(G)=(l_{uv})$, where
 $$l_{uv}=\left\{ \begin{array}{ll}
 1, & \mbox{if $u=v$ and $u$ is not an isolated vertex;}\\
 -\frac{1}{\sqrt{d_ud_v}}, & \mbox{if $u$ is adjacent to $v$;}\\
 0, &  \mbox{otherwise}.
 \end{array}\right.$$
 The eigenvalues of $\mathcal{L(G)}$ are called $\mathcal{L}-$eigenvalues. This paper deals with the normalized Laplacian matrix of $G$, so we denote its spectrum of $G$ (the all eigenvalues  $\mathcal{L}(G)$ of $G$, including multiplicities) by $Sp(G)$.  We say that $G$ and $H$ are {\it cospectral} if they are not isomorphic, but $Sp(G)=Sp(H)$, and that $G$ is determined by its normalized Laplacian spectrum if $Sp(H)=Sp(G)$ only when $H$ is isomorphic to $G$.

  Chung in \cite{fan1997} showed how the normalized Laplacian spectrum reveals  fundamental properties and structure of a graph.
   Butler  \cite{butler2015} surveyed  algebraic aspects of $\mathcal{L}(G)$ and provided (see \cite{butler2011} and \cite{butler2016}) several methods of constructing  cospectral graphs. Almost all small graphs are determined by their normalized Laplacian spectrum (see \cite{butler2016}).
    The normalized Laplacian spectrum of a complete graph $K_n$ is $0, \frac{n}{n-1}$ with multiplicities $1$ and $n-1$ respectively, and $K_n$ is determined by this spectrum \cite{butler2016}.

      Butler \cite{butler2015} conjectured that the only graphs cospectral with a cycle are $K_{1,3}$ and the graph $\gamma_{4k}$ obtained  by identifying the center vertex of a path on $2k+1$ vertices and a vertex of a cycle on $2k$ vertices; i.e.,  a cycle on   $n$ vertices is determined by its normalized Laplacian spectrum if and only if  $n>4$ and $4\nmid n$.  In general, up to now, there are very few graphs that are known to be determined by their normalized Laplacian spectrum. In this paper, we present a family of graphs that are determined by their spectrum.

   Denote by $F_{p, q}$ the graph $K_1\bigvee (pK_q)$ on $pq+1$ vertices obtained by joining a vertex to all vertices of $p$ disjoint copies of the complete graph $K_q$ on $q$ vertices.  The {\it friendship graph $F_k$} consists of $k$ edge-disjoints triangles that meet in one vertex (for the famous friendship theorem, see \cite{erdos1966}).   Liu et. al. \cite{liu2008} proved the $F_k$ is determined by its Laplacian spectrum (see also \cite{lin2010}). Wang et. al. \cite{wang2010} proved that $F_k$ is determined by its signless Laplacian spectrum. Recently, Cioab\u{a} et. al. \cite{cioaba2015} proved that $F_k$ is determined by its adjacent  spectrum  if and only if  $k\neq 16 $ , and that  for $F_{16}$, there is only one graph nonisomorphic to $F_{16}$, but with the same adjacency spectrum. The friendship graphs are a subfamily of $F_{p, q}$, since $F_k=F_{k, 2}$.

    In this paper  we show that $F_{p, q}$ is determined by its normalized Laplacian spectrum if and only if $q\ge 2$, or $q=1$ and $p\le 2$. This shows of course that the friendship graphs are determined by their normalized Laplacian spectrum.

\begin{Theorem}\label{main} If $q\ge 2$, or $q=1$ and $p\le 2$, then
  $F_{p,q} $ is determined by its normalized Laplacian spectrum. If $q=1$ and $p\ge 3$, then there is a graph $G$, not isomorphic to $F_{p,q}$ such that $Sp(G)=Sp(F_{p, q})$.
\end{Theorem}

\begin{Corollary}
The friendship graph is determined by its normalized Laplacian spectrum
\end{Corollary}

\section{Proof of Theorem~\ref{main}}
In order to prove the main results, we need the following lemmas.

\begin{Lemma}\label{sp-of Fpq}
(i). If $p=1$, then the normalized Laplacian spectrum of $F_{1, q}$ is $0, 1+\frac{1}{q}$ with multiplicities $1$ and $q$ respectively.

(ii).  If $p$ is a  positive integer with $p\ge 2$, then the normalized Laplacian spectrum of $F_{p, q}$ is $0, \frac{1}{q}, 1+\frac{1}{q}$ with multiplicities $1$, $p-1$ and $pq-p+1$ respectively.
\end{Lemma}
\begin{proof}
  (i). Since $p=1$, $F_{1, q}=K_{q+1}$ and the assertion holds by \cite{butler2016}.

   (ii). If $q=1$, then $F_{p, q}$ is the star graph $K_{1, p}$ and its normalized Laplacian spectrum is  $0, 1,2$  with multiplicities $1$, $p-1$ and $1$ respectively.  If $q\ge 2$,  then the normalized Laplacian spectrum of $K_q$ is  $0, \frac{q}{q-1}$ with multiplicities $1$ and $q-1$ respectively. By Proposition 3 in \cite{butler2016},  the normalized Laplacian eigenvalues of $F_{p, q}$ is $0, \frac{1}{q}, 1+\frac{1}{q}$ with multiplicities $1$, $p-1$ and $pq-p+1$ respectively. Hence (ii) holds.
\end{proof}

\begin{Lemma}\label{p=1q=1}
If $q=1$ and $p\le 2$, then
   $F_{p,q} $ is determined by its normalized Laplacian spectrum. If $q=1$ and $p\ge 3$, then there is a graph $G$, not isomorphic to $F_{p,q}$ such that $Sp(G)=Sp(F_{p, q})$.
  \end{Lemma}
  \begin{proof}
  If $q=1$ and $p\le 2$, clearly, $F_{p,q} $ is $K_2$  or $K_{1,2}$. So $F_{p, q}$  is determined by its normalized Laplacian spectrum. If $q=1$ and $p\ge 3$, then  by \cite{butler2016}, the normalized Laplacian spectrum of $F_{p, q}$ is the same  as that of $K_{r,s}$, where $r+s=q+1$. So there is a graph $G$, not isomorphic to $F_{p,q}$ such that $Sp(G)=Sp(F_{p, q})$.
  \end{proof}
 From now, we assume that $p\ge 2$ and $q\ge 2$. In order to characterize graphs with three normalized Laplacian eigenvalues,  for a simple connected graph $G$, denote by
$$\widehat{d}_u :=\sum_{v\sim u}\frac{1}{d_v} \ \ \mbox{for} \ u\in V(G);$$
 $$\widehat{\lambda}_{uv} :=\sum_{w\sim u, w\sim v}\frac{1}{d_w}\ \ \mbox{ for}\  u\sim v;$$ and
$$\widehat{\mu}_{uv} :=\sum_{w\sim u, w\sim v}\frac{1}{d_w}\ \ \mbox{ for}\  u\nsim v,$$
where $u\sim v (u\nsim v)$ means that $u$ and $v$ are (not) adjacent in $G$.
It follows from Theorem~1 in \cite{van2011} that the following assertion holds.
\begin{Lemma}\label{3-d-e}
Let $G$ be a simple connected graph with $m$ edges and let $q\ge 2$ be a positive integer. Then $G$ has three distinct $\mathcal{L}-$eigenvlaues $0, \frac{1}{q}, 1+\frac{1}{q}$  if and only if the
following three properties hold.

 \begin{equation}\label{3-d-e-1}  \widehat{d}_u
 = \frac{(q+1)d_u^2}{2mq^2}+\frac{(q-1)d_u}{q^2},\ \  u\in V(G);
\end{equation}
\begin{equation}\label{3-d-e-2} \widehat{\lambda}_{uv}
=\frac{(q+1)d_ud_v}{2mq^2} +\frac{q-2}{q},\ \  u\sim v;
\end{equation}

\begin{equation}\label{3-d-e-3}  \widehat{\mu}_{uv}
=\frac{(q+1)d_ud_v}{2mq^2},  u\nsim v.\end{equation}
\end{Lemma}

\begin{Lemma}\label{lemma2}
Let $G=(V(G), E(G))$ be a simple graph and let $\delta$ be the minimum degree of $G$.  If $Sp(G)=Sp(F_{p,q})$  with positive integers $p\ge 2$ and $q\ge 2$, then the following three properties  hold.

(i).   $|V(G)|= pq+1;$

(ii). $G$ is connected;

 (iii).  $2\le \delta\le q+1$.
 \end{Lemma}

\begin{proof} (i) is obvious and (ii) follows from the fact that the number of components of $G$ is equal to the multiplicity of 0.

(iii). Let $x\in V(G)$  with $d_x=\delta$.  By (\ref{3-d-e-1}) in Lemma~\ref{3-d-e}, we have
$$\frac{q-1}{q^2}d_x<\widehat{d}_x=\sum_{w\sim x}\frac{1}{d_w}\le 1.$$
Hence $d_x<\frac{q^2}{q-1}\le q+2$. So $\delta=d_x\le q+1$ since  $\delta$ is integer. By (ii), $d(x)=\delta\ge 1$. So there exists a vertex $u\in V(G)$ such that $u\sim x$. By (\ref{3-d-e-2}) in Lemma~\ref{3-d-e},  $\widehat{\lambda}_{ux}>0$.  Then there exists another vertex $w\neq u, x$ such that it is adjacent to both $u$ and $x$. So $\delta\ge 2$.
\end{proof}

\begin{Lemma}\label{lemm3}
Let $G=(V(G), E(G))$ be a simple graph  with $m$ edges and the minimum degree $\delta$ such that $Sp(G)=Sp(F_{p,q})$ with positive integers $p\ge 2$ and $q\ge 2$.  If $d_x=\delta$ for $x\in V(G)$ with the neighbor set $N(x)=\{y_1, \ldots, y_\delta\}$ and $d_{y_1}\le \cdots\le d_{y_\delta}$, then $d_{y_1}=\delta$.
\end{Lemma}

\begin{proof}
We consider the following two cases.

{\bf Case 1:} $d_{y_1}\ge \delta+2$. By (\ref{3-d-e-1}) in Lemma~\ref{3-d-e}, we have
$$\frac{(q-1)\delta}{q^2}=\frac{(q-1)d_x}{q^2}<\widehat{d}_x = \sum_{w\sim x}\frac{1}{d_w}\le \frac{\delta}{\delta+2},$$
which implies $\delta<q-1+\frac{1}{q-1}$. So $\delta\le q-1$ by $q\ge 2$.
On the other hand, by (\ref{3-d-e-2}) in Lemma~\ref{3-d-e}, we have
$$\frac{q-2}{q}<\widehat{\lambda}_{xy_1}=\sum_{w\sim x, w\sim y_1}\frac{1}{d_w}\le \frac{\delta-1}{\delta+2},$$
which implies  $\delta>\frac{3q-4}{2}$. Hence $q-1\ge\delta>\frac{3q-4}{2}$, i.e., $q<2$. A contradiction.

{\bf Case 2:} $d_{y_1}=\delta+1$. Then  $|N(x)\bigcap N(y_1)|=\delta-1$. In fact, suppose that $ |N(x)\bigcap N(y_1)|<\delta-1$.
  by (\ref{3-d-e-2}) in Lemma~\ref{3-d-e}, we have
$$\frac{q-2}{q}<\widehat{\lambda}_{xy_1}=\sum_{w\sim x, w\sim y_1}\frac{1}{d_w}\le \frac{\delta-2}{\delta+1},$$
which implies that $\frac{3q-2}{2}<\delta$.
On the other hand, by (\ref{3-d-e-1}) in Lemma~\ref{3-d-e}, we have
$$ \frac{(q-1)d_x}{q^2}<\widehat{d}_x=\sum_{i=1}^\delta\frac{1}{d_{y_i}}\le \frac{\delta}{\delta+1},$$
which implies  $\delta<q+\frac{1}{q-1}$ . So $\delta\le q$ by $q\ge 2$.  Hence $\frac{3q-2}{2}<\delta\le q$, i.e., $q<2$. This is a contradiction.
So  $|N(x)\bigcap N(y_1)|=\delta-1$, i.e, $ N(x)\bigcap N(y_1)=\{y_2, \ldots, y_\delta\}$.  Then by (\ref{3-d-e-1}) and (\ref{3-d-e-2}) in Lemma~\ref{3-d-e}, we have
\begin{equation}\label{LEMMA3-1}  \frac{(q+1)d_x^2}{2mq^2}+\frac{(q-1)d_x}{q^2}=\sum_{i=1}^\delta\frac{1}{d_{y_i}},
\end{equation}
\begin{equation}\label{lemma3-2}
\frac{(q+1)d_xd_{y_1}}{2mq^2}+\frac{(q-2)}{q}=\sum_{i=2}^\delta\frac{1}{d_{y_i}}.
\end{equation}
Hence by (\ref{LEMMA3-1}) and $\delta+1=d_{y_1}\le \cdots \le d_{y_\delta}$, we have $$\frac{(q-1)\delta}{q^2}< \sum_{i=1}^\delta\frac{1}{d_{y_i}}\le \frac{\delta}{\delta+1}.$$
Thus $\delta<q+\frac{1}{q-1}$, i.e., $\delta\le q$ by $q\ge 2$.  On the other hand, by  (\ref{lemma3-2}), we have
$\frac{q-2}{q}<\frac{\delta-1}{\delta+1}$,  which implies that $q<\delta+1$, i.e., $q\le \delta$. Hence $\delta=q\ge 2$. Furthermore,
by $d_x=\delta=q$,  subtracting (\ref{lemma3-2}) from  (\ref{LEMMA3-1}) yields
$2m=(\delta+1)^2$.  Since $d_{y_2}\ge d_{y_1}\ge \delta+1$, there exists a vertex $w\notin\{x, y_1, \ldots, y_\delta\}$ with $d_w\ge \delta\ge 2$
Hence
$ (\delta+1)^2=2m\le \delta+(\delta+1)\delta+2=(\delta+1)^2+1$, which is a contradiction.

 Therefore by Cases 1 and 2, we have $d_{y_1}=\delta$  and the assertion holds.
\end{proof}

\begin{Lemma}\label{lemm4}

Let $G=(V(G), E(G))$ be a simple graph  with $m$ edges and the minimum degree $\delta$ such that $Sp(G)=Sp(F_{p,q})$ for positive integers $p\ge 2$ and $q\ge 2$.  If $d_x=\delta$ for $x\in V(G)$ with the neighbor set $N(x)=\{y_1, \ldots, y_\delta\}$ and $d_{y_1}\le \cdots\le d_{y_\delta}$, then $N(x)\bigcap N(y_1)=\{y_2, \ldots, y_\delta\}$.
\end{Lemma}

\begin{proof} By Lemma~\ref{lemm3}, $d_{y_1}=\delta$. Now we consider the following two cases.

 {\bf Case 1:} $|N(x)\bigcap N(y_1)|\le \delta-3$.  By (\ref{3-d-e-2}) in Lemma~\ref{3-d-e}, we have
 $$\frac{q-2}{q}<\widehat{\lambda}_{xy_1}=\sum_{w\sim x, w\sim y_1}\frac{1}{d_w}\le \frac{\delta-3}{\delta},$$
 which yields that $\frac{3q}{2}<\delta$.  By Lemma~\ref{lemma2}(iii), we get $\frac{3q}{2}<\delta\le q+1$, i.e., $q<2$. This contradicts to the condition.

{\bf Case 2:} $|N(x)\bigcap N(y_1)|=\delta-2$.  By (\ref{3-d-e-2}) in Lemma~\ref{3-d-e}, we have $\frac{q-2}{q}<\frac{\delta-2}{\delta}$, which yields $q<\delta$, i.e., $q\le \delta-1$. By Lemma~\ref{lemma2}(iii), we have $\delta=q+1$.
By (\ref{3-d-e-1}) and (\ref{3-d-e-2}) in Lemma~\ref{3-d-e}, we have
\begin{equation}\label{lemma4-1}
\frac{(q+1)d_x^2}{2mq^2}+\frac{(q-1)d_x}{q^2}=\sum_{i=1}^\delta\frac{1}{d_{y_i}},
\end{equation}
\begin{equation}\label{lemma4-2}
\frac{(q+1)d_xd_{y_1}}{2mq^2}+\frac{q-2}{q}=\sum_{i=2}^\delta \frac{1}{d_{y_i}}-\frac{1}{d_{y_j}}\ \mbox {for some}\  2\le j\le \delta.
\end{equation}
By   $d_x=\delta=q+1$ and $d_{y_1}=\delta$, subtracting (\ref{lemma4-2}) from (\ref{lemma4-1}) yields
$$\frac{(q-1)(q+1)}{q^2}-\frac{q-2}{q}=\frac{1}{q+1}+\frac{1}{d_{y_j}}.$$
Hence
$d_{y_j}=q+\frac{q}{q^2+q-1}$  contradicts the fact that  $d_{y_j}$ is an integer.

Hence by Cases 1 and 2, we have $|N(x)\bigcap N(y_1)|= \delta-1$  and  the assertion holds.
\end{proof}

\begin{Lemma}\label{lemm5}

Let $G=(V(G), E(G))$ be a simple graph  with $m$ edges and the minimum degree $\delta$ such that $Sp(G)=Sp(F_{p,q})$ for positive integers $p\ge 2$ and $q\ge 2$.  If $d_x=\delta$ for $x\in V(G)$ with the neighbor set $N(x)=\{y_1, \ldots, y_\delta\}$  and $\delta=d_x=d_{y_1}=\cdots =d_{y_k}<d_{y_{k+1}}\le \cdots\le d_{y_\delta}$ for $1\le k\le \delta-1$,  then $ N(x)\bigcap N(y_i)=\{y_1, \ldots, y_{i-1}, y_{i+1}, \ldots, y_\delta\}$ for $i=1, \ldots, k$ and   $d_{y_{k+1}}\ge \delta+2$.
\end{Lemma}
\begin{proof} By Lemma~\ref{lemm4} and $\delta=d_x=d_{y_1}=\cdots =d_{y_k},$ it is easy to see that $ N(x)\bigcap N(y_i)=\{y_1, \ldots, y_{i-1}, y_{i+1}, \ldots, y_\delta\}$ for $i=1, \ldots, k$.
Now we prove that  $d_{y_{k+1}}\ge \delta+2$. Suppose that $d_{y_{k+1}}< \delta+2$. Then  $d_{y_{k+1}}=\delta+1$ by $d_{y_{k+1}}>d_{y_k}=\delta$.  We consider the following three cases.

{\bf Case 1:}  $|N(x)\bigcap N(y_{y_{k+1}})|\le \delta-3.$
By (\ref{3-d-e-2}) in Lemma~\ref{3-d-e}, we have
$$\frac{q-2}{q}<\widehat{\lambda}_{xy_{k+1}}=\sum_{w\sim x, w\sim y_{k+1}}\frac{1}{d_w}\le \frac{\delta-3}{\delta},$$
which implies that $\delta>\frac{3q}{2}$. By Lemma~\ref{lemma2} (iii), $\delta\le q+1$.  Thus  we have $q+1>\frac{3q}{2}$, i.e., $q<2$. A contradiction.

{\bf Case 2:} $|N(x)\bigcap N(y_{y_{k+1}})|= \delta-2.$ By (\ref{3-d-e-2}) in Lemma~\ref{3-d-e}, we have
$$\frac{q-2}{q}<\widehat{\lambda}_{xy_{k+1}}\le \frac{\delta-2}{\delta}.$$
Then $q<\delta$, i.e., $q+1\le \delta$. By Lemma~\ref{lemma2} (iii), $\delta\le q+1$. Thus we have $\delta=q+1$.
By (\ref{3-d-e-1}) and (\ref{3-d-e-2}) in Lemma~\ref{3-d-e}, we have
\begin{equation}\label{lemma5-1}
\frac{(q+1)d_x^2}{2mq^2}+\frac{(q-1)d_x}{q^2}=\sum_{i=1}^\delta\frac{1}{d_{y_i}}
\end{equation}
and
\begin{equation}\label{lemma5-2}
\frac{(q+1)d_xd_{y_{k+1}}}{2mq^2}+\frac{q-2}{q}=\sum_{i=1}^\delta\frac{1}{d_{y_i}}-\frac{1}{d_{y_{k+1}}}-\frac{1}{d_{j}}\ \  \mbox{for some}\  k+2\le j\le \delta.
\end{equation}
By  $d_x=\delta=q+1$ and $d_{y_{k+1}}=\delta+1=q+2$,  subtracting (\ref{lemma5-2}) from (\ref{lemma5-1}) yields
that
$$\frac{1}{d_{y_j}}=\frac{q^2+3q-2}{q^2(q+2)}-\frac{(q+1)^2}{2mq^2}.$$
Since $2m\ge d_x+d_{y_1}+\cdots+d_{y_\delta}> \delta(\delta+1)=(q+1)(q+2)$, we have
$$\frac{1}{d_{y_j}}>\frac{q^2+3q-2}{q^2(q+2)}-\frac{(q+1)^2}{(q+1)(q+2)q^2}.$$
So $d_{y_j}< q+2=\delta+1$ contradicts that $d_{y_j}\ge d_{y_{k+1}}=\delta+1$.

{\bf Case 3:} $|N(x)\bigcap N(y_{y_{k+1}})|= \delta-1.$ Then $N(x)\bigcap N(y_{y_{k+1}})=\{y_1, \ldots, y_k,$ $ y_{k+2},$ $ \ldots, y_\delta\}$.
By (\ref{3-d-e-1}) and (\ref{3-d-e-2}) in Lemma~\ref{3-d-e}, we have
\begin{equation}\label{lemma5-3}
\frac{(q+1)d_x^2}{2mq^2}+\frac{(q-1)d_x}{q^2}=\sum_{i=1}^\delta\frac{1}{d_{y_i}},
\end{equation}
\begin{equation}\label{lemma5-4}
\frac{(q+1)d_xd_{y_{k+1}}}{2mq^2}+\frac{q-2}{q}=\sum_{i=1}^\delta\frac{1}{d_{y_i}}-\frac{1}{d_{y_{k+1}}}.
\end{equation}
Subtracting (\ref{lemma5-4}) from (\ref{lemma5-3}) yields
that
\begin{equation}\label{lemma5-5}
\frac{(q-1)\delta}{q^2}=\frac{q-2}{q}+\frac{1}{\delta+1}+\frac{(q+1)\delta}{2mq^2}.
\end{equation}
By $ \delta\le q+1$ in Lemma~\ref{lemma2} and (\ref{lemma5-5}), we have
$$\frac{(q-1)\delta}{q^2}> \frac{q-2}{q}+\frac{1}{q+2},$$
which implies   $\delta> q-1$ by $q\ge 2$.  Hence $\delta=q$ or $\delta=q+1$.
 We consider the following two subcases.

 {\bf Subcase 3.1:} $\delta=q$. By (\ref{lemma5-5}), we have $$2m=(q+1)^2.$$  Furthermore, by (\ref{lemma5-3}), we have
 $$\frac{1}{q+1}+\frac{q-1}{q}=\sum_{i=1}^\delta\frac{1}{d_{y_i}}\le \frac{k}{\delta}+\frac{\delta-k}{\delta+1}.$$
 Thus $k\ge \delta-1$ and $k=\delta-1$.  Since $d_{k+1}= \delta+1$, there exists a vertex $z\in V(G)\setminus \{x, y_1, \ldots, y_{\delta}\}$ such that $z\sim y_{\delta}$.  Furthermore, $N(z) \bigcap \{x, y_1, \ldots, y_{\delta-1}\}=\emptyset$. So there exists another vertex $w\in V(G)\setminus\{x, y_1, \ldots, y_{\delta}, z\}$ by $d_z\ge \delta\ge 2$. Therefore, $2m\ge \delta^2+(\delta+1)+d_z+d_w>(\delta+1)^2=(q+1)^2$ which contradicts to $2m=(q+1)^2$.

{\bf Subcase 3.2:} $\delta=q+1$. By (\ref{lemma5-5}), we have
$$2m=\frac{(q+1)^2(q+2)}{q^2+3q-2}<2q+2.$$ It is a contradiction.
Hence  $d_{y_{k+1}}\ge \delta+2$ and the assertion holds.\end{proof}

\begin{Lemma}\label{lemma6}

Let $G=(V(G), E(G))$ be a simple graph  with $m$ edges and the minimum degree $\delta$ such that $Sp(G)=Sp(F_{p,q})$ for positive integers $p\ge 2$ and $q\ge 2$.  If $d_x=\delta$ for $x\in V(G)$ with the neighbor set $N(x)=\{y_1, \ldots, y_\delta\}$  and $\delta=d_x=d_{y_1}=\cdots =d_{y_k}<d_{y_{k+1}}\le \cdots\le d_{y_\delta}$ for some $1\le k\le \delta-1$,  then

(i) $\delta=q$.

(ii) $N(x)\bigcap N(y_{k+1})=\{y_1, \cdots, y_k, y_{k+2}, \cdots, y_\delta\}$.

(iii) $k=q-1$.

(iv) $d_{y_q}=pq$ and $d_w=q $ for $w\in V(G)\setminus\{y_q\}$.

(v). $2m=pq(q+1)$.

\end{Lemma}
\begin{proof}
(i). By Lemma~\ref{lemm5}, we have $N(x)\bigcap N(y_k)=\{y_1, \cdots, y_{k-1}, y_{k+1}, \cdots, y_\delta\}$.
 Hence by (\ref{3-d-e-1}) and (\ref{3-d-e-2}) in Lemma~\ref{3-d-e}, we have
 \begin{equation}\label{lemma6-1}\frac{(q+1)d_x^2}{2mq^2}+\frac{(q-1)d_x}{q^2}=\sum_{i=1}^\delta\frac{1}{d_{y_i}},
 \end{equation}
 \begin{equation}\label{lemma6-2}\frac{(q+1)d_xd_{y_k}}{2mq^2}+\frac{q-2}{q}=\sum_{i=1}^\delta\frac{1}{d_{y_i}}-\frac{1}{d_{y_k}}.\end{equation}
 By $d_x=d_{y_k}=\delta$, subtracting (\ref{lemma6-2}) from (\ref{lemma6-1}) yields
 $\frac{(q-1)\delta}{q^2}-\frac{q-2}{q}=\frac{1}{\delta}$. So $\delta=q$ and (i) holds.

(ii).  Suppose that $|N(x)\bigcap N(y_{k+1})|\le \delta-2$. Then by (\ref{3-d-e-2}) in Lemma~\ref{3-d-e}, we have
 $$\frac{q-2}{q}<\widehat{\lambda}_{xy_{k+1}}=\sum_{w\sim x, w\sim y_{k+1}}\frac{1}{d_{w}}\le \frac{\delta-2}{\delta}.$$
 So $q<\delta$,  which contradicts to (i). So $|N(x)\bigcap N(y_{k+1})|=\delta-1$ and (ii) holds.

 (iii). By (\ref{3-d-e-2}) in Lemma~\ref{3-d-e} and (ii), we have
 \begin{equation}\label{lemma6-3}
 \frac{(q+1)d_xd_{y_k}}{2mq^2}+\frac{q-2}{q}=\sum_{i=1}^\delta\frac{1}{d_{y_i}}-\frac{1}{d_{y_k}},\end{equation}
   \begin{equation}\label{lemma6-4}
 \frac{(q+1)d_xd_{y_{k+1}}}{2mq^2}+\frac{q-2}{q}=\sum_{i=1}^\delta\frac{1}{d_{y_i}}-\frac{1}{d_{y_{k+1}}}.\end{equation}
Subtracting (\ref{lemma6-3}) from (\ref{lemma6-4}) yields
$$\frac{(q+1)d_x(d_{y_{k+1}}-d_{y_k})}{2mq^2}=-\frac{1}{d_{y_{k+1}}}+\frac{1}{d_{y_k}}.$$
Moreover, by $d_{y_{k+1}}> d_{y_k}=d_x=\delta =q$ from (ii),  we have $2m=(q+1)d_{y_{k+1}}$. Furthermore, by (\ref{3-d-e-1}) in Lemma~\ref{3-d-e},
$$\frac{(q+1)d_x^2}{2mq^2}+\frac{(q-1)d_x}{q^2}=\sum_{i=1}^\delta\frac{1}{d_{y_i}}\le \frac{k}{\delta}+\frac{\delta-k}{d_{y_{k+1}}}.$$
By $2m=(q+1)d_{y_{k+1}}$ and $d_x=\delta=q$,
 we have $$ \frac{1}{d_{y_{k+1}}}+\frac{q-1}{q}\le  \frac{k}{q}+\frac{q-k}{d_{y_{k+1}}},$$
i.e.,  $ (q-k-1)(d_{y_{k+1}}-q)\le 0$. So $k\ge q-1=\delta-1$.  Hence $k=\delta-1=q-1$ and (iii) holds.

 (iv). By (iii), $2m=(q+1)d_{y_\delta}$. Moreover, by Lemma~\ref{lemma2} (i), $|V(G)|=pq+1$. Hence
\begin{equation}\label{lemma6-5}
(\delta+1)d_{y_\delta}=2m =d_{y_\delta}+\sum_{w\neq y_\delta}d_w\ge d_{y_\delta}+\delta(pq).\end{equation}
 Hence $d_{y_\delta}\ge pq$. However, $d_{y_\delta}\le |V(G)|-1=pq$. So $d_{y_\delta}=pq$. Furthermore, (\ref{lemma6-5}) becomes equality, which implies $d_w=\delta$ for $w\neq y_\delta$. So (iv) holds.

 (v). By (i), (iii) and (iv), we  have  $q=\delta$, $2m=(q+1)d_{y_\delta}$ and $d_{y_\delta}=pq$, which implies $2m=(q+1)pq$.
\end{proof}

 Now we are  ready to prove the main theorem in this paper.

 {\bf Proof   of Theorem~\ref{main}.} If $p=1,$ then $F_{1,q}=K_{q+1}$, so it  is determined by its normalized Laplacian spectrum.
If $q=1$ and $p\le 2$, or $q=1$ and $p\ge 3$, the assertion follows from Lemma~\ref{p=1q=1}.

  Now we assume that $p\ge 2$ and $q\ge 2$.
Let $G$ be any graph with $Sp(G)=Sp(F_{p, q})$. Then $G$ is a connected graph on $pq+1$ vertices by Lemma~\ref{lemma2}. Furthermore let $d_x=\delta(G)$ for $x\in V(G)$ and $N(x)=\{y_1, \ldots, y_\delta\}$ with $d_{y_1}\le \cdots\le d_{y_\delta}$.
By Lemma~\ref{lemm3}, $d_{y_1}=\delta$. Hence we assume that $\delta=d_{y_1}=\cdots=d_{y_k}<d_{y_{k+1}}\le \cdots\le d_{y_\delta}$,
$1\le k\le \delta$. Then $1\le k\le \delta-1$, otherwise $G=K_{\delta+1}$ has only two distinct $\mathcal{L}-$ eigenvalues. Hence by Lemma~\ref{lemma6}, we have $\delta=q$, $d_{y_1}=\cdots=d_{y_{q-1}}=q$ and  the induced subgraph $G[x, y_1, \cdots, y_{q-1}]$ by vertex set $\{x, y_1, \cdots, y_{q-1}\}$ is a clique of order $q$.  Furthermore, $y_q$ is adjacent to each vertex in $G$.  For any $u\in V(G)\setminus\{x, y_1, \ldots, y_q\}$, then $u$ is not adjacent to $x, y_1, \ldots, y_{q-1}$.  By (\ref{3-d-e-3}) in Lemma~\ref{3-d-e}, we have
$$
\widehat{\mu}_{ux}=\sum_{w\sim u, w\sim v}\frac{1}{d_w}=\frac{1}{d_{y_q}} = \frac{(q+1)d_ud_x}{2mq^2}. $$
By Lemma~\ref{lemma6}, $d_u=q=\delta$, i.e., the degree of each vertex except $y_q$ in $V(G)$ is $q.$   Choosing any vertex $v\in V(G)\setminus \{x, y_1, \ldots, y_q\}$, let $N(v)=\{z_1, \ldots, z_{q-1}, y_q\}$ since $d_v=q$ and $y_q\sim v$.  Then $d_{z_1}=\cdots =d_{z_{q-1}}=q$ and $N(z)\bigcap \{x, y_1, \ldots, y_q\}=\{y_q\}$.
By (\ref{3-d-e-2}) in Lemma~\ref{3-d-e}, we have
$$\frac{(q+1)d_vd_{z_i}}{2mq^2}+\frac{q-2}{q}=\sum_{w\sim v, w\sim z_i}\frac{1}{d_w}=\frac{1}{d_{y_q}}+\frac{|N(v)\bigcap N(z_i)\setminus\{y_q\}|}{q}, $$
for $v\sim z_i$.
Since $2m=pq(q+1)$, $d_v=d_{z_i}=q$ and $d_{y_q}=pq$, we have $|N(v)\bigcap N(z_i)\setminus\{y_q\}|=q-1$. Hence  the induced subgraph $G[v,z _1, \ldots, z_{q-1}]$ by vertex set $\{v, z_1, \ldots, z_{q-1}\}$ is a clique of order $q$.
By repeating the above  process, $G$ must be $F_{p, q}$. So the assertion holds.\hfill $\square$

\noindent {\bf Acknowledgement.} The authors would like to thank the referee for very constructive suggestions and comments on this paper.


\begin{thebibliography}{12}


\bibitem{butler2015}S.~Butler, Algebraic aspects of the normalized Laplacian. {\it Recent trends in combinatorics, 295-315, IMA Vol. Math. Appl., 159,} Springer, 2016
    
    \bibitem{butler2011}S.~Butler and J.~Grout,  A construction of cospectral graphs for the normalized Laplacian. {\it Electron. J. Combin.}  18  (2011),  no. 1, Paper 231, 20 pp.

    

    \bibitem{butler2016}S.~Butler and K.~Heysse, A cospectral family of graphs for the normalized Laplacian found by toggling. {\it Linear Algebra Appl.}  507(2016) 499-512.

\bibitem{fan1997} F. Chung, Spectral Graph Theory, 2nd edn. AMS, Providence, 1997.



        \bibitem{cioaba2015}S.~M.~Cioab\u{a}, W.~H.~Haemers, J.~Vermette and W.~ Wong, The graphs with all but two eigenvalues equal to  $\pm1$,
        {\it J. Algebr. Comb.}  41(2015) 887-897
\bibitem{erdos1966}P.~Erd\H{o}s, A.~R\'{e}nyi and V.~S\'{o}s,  On a problem of graph theory. {\it Studia Sci. Math. Hung.} 1 (1966) 215-235.


 \bibitem{lin2010}Y.-Q.~Lin, J.-L.~Shu and Y.~Meng,  Laplacian spectrum characterization of extensions of vertices of wheel graphs and multi-fan graphs. {\it Comput. Math. Appl.}  60  (2010) 2003-2008.



 \bibitem{liu2008}X.-G.~Liu,Y.-P.~Zhang and X.-Q.~Gui,  The multi-fan graphs are determined by their Laplacian spectra. {\it Discrete Math.}  308  (2008) 4267-4271.

     \bibitem{van2003}   E. R. Van Dam and W. H. Haemers, Which graphs are determined by their spectrum?, {\it Linear Algebra Appl.,} 373
(2003) 241-272.
\bibitem{van2009}E. R. Van Dam and W. H. Haemers, Developments on spectral characterizations of graphs, {\it Discrete Math.,} 309 (2009)
576-586.
\bibitem{van2011}E.~R.~Van Dam and G.~R.~Omidi, Graphs whose normalized Laplacian has three eigenvalues. {\it Linear Algebra Appl.}  435  (2011) 2560-2569.

\bibitem{wang2010}J.-F.~Wang, F.~Belardo, Q.-X.~Huang and B.~Borovicanin, On the two largest Q-eigenvalues of graphs.
{\it Discrete. Math.} 310 (2010) 2858-2866.


\end{thebibliography}
\end{document}